\documentclass[12pt,reqno]{amsart}
\usepackage{amsmath,amssymb,latexsym,textcomp,mathrsfs}
\usepackage[all]{xy}
\usepackage{graphicx}
\usepackage{bm,amsmath, amsthm, amssymb, amsfonts}
\usepackage{times}
\usepackage[english]{babel}

\setbox0=\hbox{$+$}
\newdimen\plusheight
\plusheight=\ht0
\def\+{\;\lower\plusheight\hbox{$+$}\;}

\setbox0=\hbox{$-$}
\newdimen\minusheight
\minusheight=\ht0
\def\-{\;\lower\minusheight\hbox{$-$}\;}

\setbox0=\hbox{$\cdots$}
\newdimen\cdotsheight
\cdotsheight=\plusheight
\def\cds{\lower\cdotsheight\hbox{$\cdots$}}

\numberwithin{equation}{section}
\theoremstyle{plain}
\newtheorem{theorem}{Theorem}[section]
\newtheorem{lemma}{Lemma}[section]
\newtheorem{example}{Example}[section]

\newtheorem{definition}{Definition}[section]

\newtheorem{remark}{Remark}[section]
\newtheorem{note}{Note}[section]

  \newenvironment{nouppercase}{%
   \renewcommand{\uppercasenonmath}[1]{}}{}
	
	 \newcommand{\Keywords}[1]{\par\noindent
   {\small{Keywords and phrases}: #1}}
   
   \newcommand{\AMS}[1]{\par\noindent
   {\small{AMS Subject Classification (2010)}: #1}}

\begin{document}

\title{$s\beta_\lambda $-CLOSED SETS AND SOME LOW SEPARATION AXIOMS IN GT-SPACES}

 \author{Jagannath Pal$^1$}
 \author{Amar Kumar Banerjee$^2$}
 \newcommand{\acr}{\newline\indent}
 \maketitle
 \address{{1\,} Department of Mathematics, The University of Burdwan, Golapbag, East Burdwan-713104,
 West Bengal, India.
 Email: jpalbu1950@gmail.com\acr
 {2\,} Department of Mathematics, The University of Burdwan, Golapbag, East Burdwan-713104,
 West Bengal, India. Email: akbanerjee@math.buruniv.ac.in, akbanerjee1971@gmail.com 
 \\} 
   
\begin{abstract}
Here we have studied the ideas of $ sg_\lambda,s\lambda$ and $ s\beta_\lambda $-closed sets and investigated some of their properties in generalized topological spaces. We have also studied some low separation axioms  namely $ s\lambda T_\frac{1}{4} $,  $ s\lambda T_\frac{3}{8} $,  $ s\lambda T_\frac{1}{2} $ axioms and their mutual relations with $ s\lambda T_0 $ and $s\lambda T_1 $ axioms. 
\end{abstract}

\begin{nouppercase}
\maketitle
\end{nouppercase}

\let\thefootnote\relax\footnotetext{
\AMS{Primary 54A05, 54A10, 54D10}
}

\section{\bf Introduction}
\label{sec:int}
After generalization of a topological space given by A.D. Alexandroff \cite{ALD} in 1940 to $ \sigma $-space (Alexandroff space), many topologists turned their attention to carry out their works in such direction viz. in $ \sigma $-spaces and bispaces etc. \cite{AD, AS, BS, BP, LD} where several works in respect of topological properties have been studied. On the other hand, N.Levine  \cite{NL} 1970 introduced the concept of generalized closed sets (g-closed sets) in a topological space which opened the door of many aspects in topological spaces to generate different types of generalized sets. In 1987, Bhattacharyya and Lahiri \cite{BL} introduced the class of semi-generalizsed closed sets in a topological space. Since then many authors have contributed to the subsequent development of various topological properties on semi generalized closed sets ($ sg $-closed sets)  \cite{DM}, \cite{MBD}, \cite{SM} where many more references are found. By taking an equivalent form of $ g $-closed sets,  M. S. Sarsak  \cite{MS} introduced $ g_\mu $-closed sets in a generalized topological space $ (X,\mu) $  and studied new separation axioms namely $ \mu$-$T_\frac{1}{4},  \mu$-$T_\frac{3}{8} $ and $ \mu$-$T_\frac{1}{2} $ axioms by defining $ \lambda_\mu $-closed sets and investigated their properties and relations among the new axioms with $ \mu$-$T_0 $ and $ \mu$-$T_1 $ axioms. It can be verified that the properties given by Sarsak \cite{MS}   almost remain same if we take $ sg_\mu $-closed sets and $ s\lambda_\mu $-closed sets  in generalized topological spaces.

We have extended the notion of $ g_\mu $-closed sets and $ \lambda_\mu $-closed sets in a more general structure of generalized topological space by introducing the idea of $ s\lambda $-closed sets, $sg_\lambda$-closed sets and $s\beta_\lambda$-closed sets and investigated some of their properties. We also give $ s\lambda  T_\frac{1}{4},  s\lambda T_\frac{3}{8}, s\lambda T_\frac{1}{2} $ axioms  and explore mutual relations among these axioms with $ s\lambda T_0$ and $ s\lambda T_1$. Moreover, we have studied $ s\lambda $-homeomorphism in a generalized topological space.

 \section{\bf Preliminaries}
 \label{sec:pre}
 
 Let $ X $ be a nonempty set. A generalized topology $\mu $ \cite{MS} is a collection of subsets of $ X $ such that $ \emptyset \in  \mu $ and $ \mu $ is closed under arbitrary unions. In a generalized topological space (in short a GTspace) $ (X, \mu) $, members of $ \mu $ are called $ \mu $-open sets and complements are $ \mu $-closed sets. When there is no confusion, the GTspace  $(X , \mu)$ will simply be denoted by $ X $.
 
\begin{definition} (c.f.\cite{BN})\label{1}. Let $ (X,\mu) $ be a GTspace then
 
(1) a point $ x\in X $ is said to be a $ \mu $-adherence point of a subset $ A$ of $ X $, if for every $ \mu $-open set $ U $ containing $ x $ such that $ A\cap U\not=\emptyset $. The set of all $ \mu $-adherence points of $ A $ is called $ \mu $-closure of $ A $ and is denoted by $ \overline{A_\mu} $.
 
(2)  $ \mu $-interior of a subset $ A $ of $ X $ is defined as the union of all $ \mu $-open sets contained in $ A $ and is denoted by $ Int_\mu(A) $.
 \end{definition}
 
\begin{lemma}\label{2}(c.f.\cite{BC}).
Let $ A,B$ be subsets of $ X $. For $ \mu $-closure the following hold:
 
(1) $ \overline{A_\mu} =\bigcap\{F:A\subset F; F $ is $ \mu $-closed\}. 
  
(2) $ \overline{A_\mu}$  is $ \mu $-closed.
 
(3) $ A $ is $ \mu $-closed if and only if $ A=\overline{A_\mu} $.
 
(4) $ A \subset\overline{A_\mu} $ and $\overline{ \overline{(A_\mu)}_\mu}= \overline{A_\mu} $.
 
(5) If $ A\subset B $, then $ \overline{A_\mu} \subset \overline{B_\mu}  $.
 \end{lemma}

Clearly for a subset $ A\subset X, x\in \overline{A_\mu} $ iff each $ \mu $-open set containing $ x $ intersects $ A $.

\begin{remark}\label{3}(c.f.\cite{BC}).
Let $ A,B $ be subsets of $ X $, then for $ \mu $-interior the following hold:
 
(1)  $ Int_\mu(A) \subset A $. \quad\qquad \qquad  \qquad \qquad (2) If $ A\subset B $, then  $ Int_\mu(A)  \subset  Int_\mu(B) $
 
(3) $  Int_\mu(A) $ is $ \mu $-open. \qquad \qquad \qquad \quad (4) $ A $ is $ \mu $-open if and only if $  Int_\mu(A) =A$.
 \end{remark}

\begin{definition} \label{4}(c.f.\cite{LD}).
A set $ A $ in $ X $ is said to be semi $ \mu $-open ($ s\mu $-open) if there exists a $ \mu $-open set $ E $ in $ X $ such that $ E\subset A \subset \overline{E_\mu} $ i.e. $ A\subset \overline{(Int_\mu(A))_\mu} $. A set $ A $ is  semi $ \mu $-closed ($ s\mu $-closed) if and only if $ X - A $ is semi $ \mu $-open.
\end{definition}

\begin{definition}\label{5} (c.f.\cite {MS}). Let $A $ be a subset of $ (X,\mu) $. We define $ sA_\mu^\wedge= \bigcap \{ U: A\subset U, U $ is $ s\mu $-open\}  and 
$ sA_\mu^\vee =  \bigcup \{F: F \subset A, F $ is $ s\mu $-closed\}. We denote $ sA_\mu^\wedge= X$ if there is no $ s\mu $-open set containing $ A $ and $ sA_\mu^\vee = \emptyset$ if there is no $ s\mu $-closed set contained in $ A $. Again $A$ is called a $ s\wedge_\mu$-set if $A = sA_\mu^\wedge$ and 
$A$ is called a $ s\vee_\mu$-set if $A = sA_\mu^\vee$.  \end{definition}

\begin{definition}\label{6}(c.f.\cite{MS}).
A subset $A$ of $ (X,  \mu) $  is said to be $ s\lambda$-closed if $ A = K\cap P $  where $ K $ is a $ s\wedge_\mu $-set   and $ P $ is a $ s\mu $-closed set. $ A $ is called  $ s\lambda$-open if $ X -A $ is  $ s\lambda $-closed.
\end{definition}

Definitions of $ s\mu $-closure, $ s\lambda $-closure, $ s\mu $-interior, $ s\lambda $-interior of a set $ A $ of $ X $ can be given as  definition \ref{1} and same may be denoted by $ \overline{sA_\mu} $, $ \overline{sA_\lambda} $, $ sInt_\mu (A) $, $ sInt_\lambda (A) $ respectively.

Obviously in a GTspace $ (X,\mu) $, a $ \mu $-open set is  $ s\mu $-open and a $ \mu $-closed set is $ s\mu $-closed, but converses may not be always true.  It can be shown easily that arbitrary union of $ s\mu $-open sets is $ s\mu $-open. Hence collection of $ s\mu $-open sets forms a generalized topology in $ X $.

\begin{lemma}\label{7}(c.f.\cite{BC}).
Let $ A,B $ be subsets of $ X $. For $ s\mu $-closure the following hold:
 
(1) $ \overline{sA_\mu} =\bigcap\{F:A\subset F; F $ is $ s\mu $-closed\}. 
  
(2) $ \overline{sA_\mu}$  is $ s\mu $-closed.
 
(3) $ A $ is $ s\mu $-closed if and only if $ A=\overline{sA_\mu} $.
 
(4) $ A \subset\overline{sA_\mu} $ and $\overline{ s\overline{(sA_\mu)}_\mu}= \overline{sA_\mu} $.
 
(5) If $ A\subset B $, then $ \overline{sA_\mu} \subset \overline{sB_\mu}  $.
 \end{lemma}

Clearly for a subset $ A\subset X, x\in \overline{sA_\mu} $ iff  any $ s\mu $-open set containing $ x $ intersects $ A $.

\begin{remark}\label{8}(c.f.\cite{BC})
Let $ A,B $ be subsets of $ X $, then for $ s\mu $-interior the following hold:
 
(1)  $ sInt_\mu(A) \subset A $. \quad\qquad \qquad  \qquad \qquad (2) If $ A\subset B $, then  $ sInt_\mu(A)  \subset  sInt_\mu(B) $
 
(3) $  sInt_\mu(A) $ is $ s\mu $-open. \qquad \qquad \qquad \quad (4) $ A $ is $ s\mu $-open if and only if $  sInt_\mu(A) =A$.
 \end{remark}

\begin{lemma}\label{9} (c.f.\cite {MS}). Suppose $A , B$ are
 subsets of $X$.  Then the  following hold:

(1)  $s\emptyset_\mu^\wedge = \emptyset $, \qquad      $s\emptyset_\mu^\vee = \emptyset$, \qquad$sX_\mu^\wedge = X$,\qquad$sX_\mu^\vee = X$

(2) $A \subset sA_\mu^\wedge$,   \qquad $A_\mu^\vee \subset A$

(3)  $s(sA_\mu^\wedge)_\mu^\wedge = sA_\mu^\wedge$,\qquad  $s(sA_\mu^\vee)_\mu^\vee = sA_\mu^\vee$.

(4)  $A \subset B \Rightarrow sA_\mu^\wedge \subset s B_\mu^\wedge$, \qquad 
$A \subset B \Rightarrow sA_\mu^\vee \subset sB_\mu^\vee$.

(5)    $s(X\backslash A)_\mu^\wedge = X\backslash sA_\mu^\vee$,  \qquad $s(X\backslash A)_\mu^\vee = X\backslash sA_\mu^\wedge$.
\end{lemma}

Note that if $ \mu $  is replaced by $ \lambda $ in the definition \ref{5}, lemma \ref{7} and remark \ref{8}  in a GTspace $ (X,\mu) $  then the similar properties of lemma \ref{7}, remark \ref{8} and  lemma \ref{9} can also be proved for the $ s\lambda $-open sets.

A set $ A $ of $ (X,\mu) $ is a $s\vee_\mu$-set if and only if $X - A$ is a $s\wedge_\mu $-set.

Clearly collection of all $ s\vee_\mu $-sets in a GTspace  $(X, \mu)$ forms a generalized topology.

\section{\bf $ s\lambda$-closed sets and $ sg_\lambda $-closed sets in GTspace and $ s\lambda T_\frac{1}{2} $ GTspace}

In this section we will discuss some properties of $ s\lambda$-closed sets, $ s\lambda $-open sets, $sg_\lambda$-closed sets and $ s\lambda T_\frac{1}{2} $ axiom in a GTspace which will be very much useful in the sequel.

\begin{definition}\label{10} (c.f.\cite{NL}). A subset $ A $ of $ X $ is said to be a $sg_\lambda$-closed set  if $ \overline{sA_\lambda}\subset U $ whenever $ A\subset U $ and $ U $ is $ s\lambda $-open. $A$ is called $sg_\lambda$-open  if $ X - A $ is $s g_\lambda$-closed.
\end{definition}

\begin{note}\label{11}
Clearly, a set $ A $ of $ X $ is $ sg_\lambda $-closed if and only if $ \overline{sA_\lambda}\subset sA_\lambda^\wedge $.
\end{note}

If $ A\subset X $, then it is easily verfied that $ X - \overline{s(X -A)}_\lambda = sInt_\lambda(A) $.

\begin{theorem}\label{12}
A subset $ A $ of $ X $ is $ sg_\lambda $-open if and only if $ F\subset sInt_\lambda (A) $ whenever $ F\subset A $ and $ F $ is $ s\lambda $-closed (or equivalently, $ sA_\lambda^\vee \subset sInt_\lambda (A) $).
\end{theorem}
\begin{proof}
Suppose $ A $ is $ sg_\lambda $-open and $ F\subset A ,  F $ is $ s\lambda $-closed. Then $ X-A\subset X-F $ where $ X-F $ is a $s\lambda $-open set and $ X-A $ is $ sg_\lambda $-closed. Then we have $ X- sInt_\lambda (A)=\overline{s(X-A)_\lambda}\subset X-F$, by definition and hence $ F\subset sInt_\lambda (A) $.

Conversely, suppose that $ X-A \subset U, U$
is $ s\lambda $-open. Then $ X-U\subset A $ and $ X-U $ is $ s\lambda $-closed. By assumption, $ X-U\subset sInt_\lambda (A) $ and so $ \overline{s(X-A)_\lambda}=X- sInt_\lambda (A)\subset U$ and hence $ X-A$ is $sg_\lambda $-closed. This implies $ A $ is $ sg_\lambda $-open.
\end{proof}

\begin{lemma}\label{13}
For $ A\subset X $, the following hold:

(1) $ A $ is $ s\lambda $-closed if and only if $
 A=sA_\mu^\wedge \cap \overline{sA_\mu} $ 

(2)  If $ A $ is $ s\lambda $-closed then $ A=sA_\mu^\wedge \cap \overline{sA_\lambda} $.

(3) If $ A $ is $ s\mu $-closed then $ A $ is $ s\lambda $-closed.
\end{lemma}

\begin{remark}\label{14}
Obviously in $ (X,\mu) $, every $s \wedge _\mu $-set  is $ s \lambda$-closed and $ s\mu $-closed set is $ s \lambda$-closed. But converses are not in general true as revealed from the  examples \ref{14} (i) and \ref{14} (ii). Clearly $ s \lambda$-closed set is $ sg_ \lambda$-closed but it is not reversible as shown by example \ref{14} (iii).  
\end{remark}

\begin{example}\label{15} (i):
Suppose $ X=\{a,b,c\}, \mu=\{\emptyset, \{a\}, \{a,b\}\} $. Then $ (X,\mu) $ is a GTspace but not a topological space. Now $ \{b\} $ is $ s\lambda $-closed as $ \{b\}=\{a,b\}\cap \{b,c\} $ where $ \{a,b\} $ is $ s\wedge_\mu $-set and $ \{b,c\}$ is  $ s\mu $-closed. But $ \{b\} $ is not $ s\wedge_\mu $-set since $ s\{b\}_\mu^\wedge=\{a,b\} $.

(ii):
Suppose $ X=\{a,b,c\}, \mu=\{\emptyset, \{a,b\}, \{b,c\}, X\} $. Then $ (X,\mu) $ is a GTspace but not a topological space. Then $ \{a,b\} $ is $ s\lambda $-closed as $ \{a,b\}=\{a,b\}\cap \{a,b,c\} $ where $ \{a,b\} $ is $ s\wedge_\mu $-set and $ \{a,b,c\}$ is  $ s\mu $-closed. But $ \{a,b\} $ is not $ s\mu$-closed since $ \{c\} $ is not $ s\mu $-open.

(iii):
Let $ X=\{a,b,c,d\}, \mu=\{\emptyset, \{a,b,c\}\} $. Then $ (X,\mu) $ is a GTspace but not a topological space. Take the subset $ \{a,d\}=B $ say. Since $ B $ is not $ s\mu $-open and not $ s\mu $-closed; so $ B $ is not $ s\lambda $-open. Subsets containing $ B $ are only $ \{a,b,d\} $ and $ \{a,c,d\} $ which are not $ s\mu $-open and not $ s\mu $-closed, hence these are not $ s\lambda $-open sets. So only $ s\lambda $-open set containing $ B $ is $ X $ and $ \overline{sB_\lambda}\subset X $. So $ B $ is $ sg_\lambda $-closed. 
Now it can be easily shown that $ sB_\mu^\wedge\cap \overline{sB_\mu}=X\cap X=X \not=B$, hence $ B $ is not $ s\lambda $-closed.
\end{example}

\begin{definition}\label{16}(c.f.\cite{MS}). Let $ (X,\mu) $ be a GTspace.
A set $A$ of $ X $ is called a semi generalised $\wedge_\lambda$-set  (in short $sg\wedge_\lambda$-set) if $sA_\lambda^\wedge\subset F$ whenever $F\supset A$ and $F$ is $ s\lambda $-closed. A set $A$ is called a semi generalized $ \vee_\lambda$-set  (in short $s g\vee_\lambda $-set) if $ X - A  $ is $s g\wedge_\lambda$-set.
\end{definition}

Clearly if a set $ A $ of $ X $ is a $ s\wedge_\lambda $-set (resp. $ s\vee_\lambda $-set), then $ A $ is $sg\wedge_\lambda$-set (resp. $ sg\vee_\lambda $-set).

\begin{theorem}\label{17} (1):
For each $ x \in X $, $ \{x\} $ is either $ s\lambda $-open or $ sg\vee_\lambda $-set.

(2): For each $ x \in X $, $ \{x\} $ is either $ s\lambda $-closed or $ \{x\} $ is $ sg_\lambda $-open.
\end{theorem}

\begin{proof} (1):
Suppose $ x\in X $ and $ \{x\} $  is not $ s\lambda $-open, then $ X-\{x\} $ is not $ s\lambda $-closed. So  $\overline{ s(X - \{x\})_\lambda}= X $. Then  $s( X-\{x\})_\lambda^\wedge\subset \overline{ s(X - \{x\})_\lambda} $ and hence $ X-\{x\} $ is $s g\wedge_\lambda $-set and so \{x\} is $s g\vee_\lambda $-set, by definition.

(2): Suppose $ x\in X $ and $ \{x\} $  is not $ s\lambda $-closed, then $ X-\{x\} $ is not $ s\lambda $-open. So  $ s(X - \{x\})_\lambda^\wedge= X $. Then  $s( X-\{x\})_\lambda^\wedge\supset \overline{ s(X - \{x\})_\lambda} $ and hence $ X-\{x\} $ is a $s g_\lambda $-closed set, by note \ref{11} and so \{x\} is a $s g_\lambda $-open set.
\end{proof}

\begin{theorem}\label{18}
For a subset $ A $ of a GTspace $(X,\mu)$ the following  hold:

(1) If  $ A $ is $ s\lambda $-closed, then $ A $ is $ sg_\lambda $-closed

(2) If $ A $ is $ sg_\lambda $-closed and $ s\lambda $-open, then $ A $ is $ s\lambda $-closed

(3) If $ A $ is $ sg_\lambda $-closed and $ A\subset B \subset  \overline{sA_\lambda}$, then $ B $ is $ sg_\lambda $-closed.
\end{theorem}

\begin{proof}
(1): This is obvious by definition.

(2): Let $ A $ be  $ sg_\lambda $-closed and $ s\lambda $-open. Then $ \overline{sA_\lambda}\subset A $ and so $ A $ is $ s\lambda $-closed.

(3): Let $ B\subset U, U $ is $ s\lambda $-open, then $ A\subset U $. By assumption $ A $ is $ sg_\lambda $-closed, then $ \overline{sA_\lambda}\subset U $. Again $ A\subset B $ implies $ \overline{sA_\lambda}\subset \overline{sB_\lambda} \subset \overline{s(\overline{sA_\lambda}})_\lambda=\overline{sA_\lambda} $, then $ \overline{sA_\lambda} =  \overline{sB_\lambda}$ and hence $ \overline{sB_\lambda} \subset U $. So $ B $ is $ sg_\lambda $-closed.
\end{proof}

\begin{theorem}\label{19}  Let $ (X,\mu) $ be a GTspace and $ A\subset X $ then

(1) $ A $  is $ sg_\lambda $-closed if and only if  $ \overline{sA_\lambda} - A $ does not contain any non-void $ s\lambda $-closed set.

(2) Let  $ A$ is $ s\wedge_\lambda $-set (resp. $ s\vee_\lambda $-set), then $ A $ is $ sg_\lambda $-closed (resp. $ sg_\lambda $-open ) if and only if $ A $ is $ s\lambda $-closed (resp. $ s\lambda $-open),

(3)  if $ sA_\lambda^\wedge $ is $ sg_\lambda $-closed (resp. $ sA_\lambda^\vee $ is $ sg_\lambda $-open),  then $ A $ is $ sg_\lambda $-closed (resp. $ sg_\lambda $-open),
\end{theorem}

\begin{proof} (1): Suppose  $ A $  is a $ sg_\lambda $-closed set and $ P $ is a non-empty $ s\lambda $-closed set and $ P\subset  \overline{sA_\lambda} - A $. Now $ A\subset X-P $, a $ s\lambda $-open set, so $ \overline{sA_\lambda}\subset X-P $, by definition. This implies that $ P\subset X-\overline{sA_\lambda} $. Thus $ P\subset \overline{sA_\lambda}\cap (X-\overline{sA_\lambda})=\emptyset $.

Conversely, let the condition hold and we have to prove that $ A $  is $ sg_\lambda $-closed. Let $ A\subset V $ where $ V $ is $ s\lambda $-open. If $ \overline{sA_\lambda} \not\subset V$ then $ \overline{sA_\lambda}\cap (X-V) $ is a non-void $ s\lambda $-closed set contained in $\overline{sA_\lambda}-A$, a contradiction to the supposition. So $ A $ is $ sg_\lambda $-closed.

(2): Let $ A $ be a $ s\wedge_\lambda $-set and $ sg_\lambda $-closed set. Then by note \ref{11}, $ \overline{sA_\lambda} \subset sA_\lambda^\wedge =A $, so $ \overline{sA_\lambda} = A $ and hence $ A $ is $ s\lambda $-closed. On the other hand, a $ s\lambda $-closed set is obviously a $ sg_\lambda $-closed. Hence the result follows. 

Again, let $ A $ be a $ s\vee_\lambda $-set and $ sg_\lambda $-open set. Then by theorem \ref{13}, $ sInt_\lambda(A)\supset  sA_\lambda^\vee =A $, so $ sInt_\lambda(A) = A $ and hence $ A $ is $ s\lambda $-open. On the other hand, a $ s\lambda $-open set is obviously a $ sg_\lambda $-open. Hence the result follows.

(3): Suppose $ sA_\lambda^\wedge $ is $ sg_\lambda $-closed. Since the set $sA_\lambda^\wedge $ is a $ s\wedge_\lambda $-set by lemma \ref{9} (3), then from (2) we get $ sA_\lambda^\wedge $ is $ s\lambda $-closed. Since $ A\subset sA_\lambda^\wedge $ which implies $ \overline{s A_\lambda}\subset \overline{s(sA_\lambda^\wedge)_\lambda}=sA_\lambda^\wedge$. Hence $ A $ is $ sg_\lambda $-closed by note \ref{11}. 

Again, suppose $ sA_\lambda^\vee $ is $ sg_\lambda $-open. Since the set $sA_\lambda^\vee $ is a $ s\vee_\lambda $-set by lemma \ref{9} (3), then from (2) we get $ sA_\lambda^\vee $ is $ s\lambda $-open. Since $ A\supset sA_\lambda^\vee $ which implies $ sInt_\lambda (A)\supset sInt_\lambda(sA_\lambda^\vee)=sA_\lambda^\vee$. Hence $ A $ is $ sg_\lambda $-open by theorem \ref{12}.
\end{proof}

\begin{lemma}\label{20}
In a GTspace $(X,\mu)$,
arbitrary intersection of $ s\wedge_\mu $-sets is a $ s\wedge_\mu $-set.
\end{lemma}
\begin{proof}
Let  $ A_\alpha, \alpha\in \Delta, \Delta $ being an index set, be an arbitrary collection of $ s\wedge_\mu $-sets. Then for each $ \alpha, A_\alpha=s(A_\alpha)_\mu^\wedge $. Let $ A=\bigcap\{A_\alpha; \alpha\in\Delta\} $. So $ A_\alpha \supset A $ for each $ \alpha\in\Delta $. Therefore $ s(A_\alpha)_\mu^\wedge \supset sA_\mu^\wedge$ for all $ \alpha\in\Delta $. This implies that $ sA_\mu^\wedge \subset  \bigcap s(A_\alpha)_\mu^\wedge=\bigcap A_\alpha=A\subset sA_\mu^\wedge$. Hence $ A= sA_\mu^\wedge $.
\end{proof}

\begin{theorem}\label{21} Suppose $ (X,\mu) $ is a GTspace then the followin hold:

(1) If $ A_i, i\in I$,   are $ s\lambda $-closed sets of $ X $, $ I $ being an index set, then $ \bigcap_iA_i $ is $ s\lambda $-closed.

(2) If $ A_i, i\in I $, are $ s\lambda $-open sets of $ X $, $ I $ being an index set, then $ \bigcup_iA_i $ is $ s\lambda $-open.
\end{theorem}
\begin{proof}
(1) Suppose $ A_i, i\in I $, are $ s\lambda $-closed sets, $ I $ being an index set, then for each $ i $ there exist a $ s\wedge_\mu $-set $ K_i $ and a $ s\mu $-closed set $ P_i $ such that $ A_i= K_i\cap P_i $. So we get $ \bigcap A_i=\bigcap (K_i\cap P_i)=(\bigcap K_i)\bigcap (\bigcap P_i) $. By lemma \ref{20} above, $ \bigcap K_i $ is a $ s\wedge_\mu $-set and $ \bigcap P_i $ is a $ s\mu $-closed set. This shows that $ \bigcap A_i $ is $ s\lambda $-closed.

(2) Suppose $ A_i, i\in I$, are $ s\lambda $-open sets, $I $ being an index set. Then $ X-A_i $ is $ s\lambda $-closed set for each i and $ X-\bigcup A_i=\bigcap (X-A_i) $. Therefore by (1), $ X-\bigcup A_i $ is $ s\lambda $-closed and hence $ \bigcup A_i $ is $ s\lambda $-open.
\end{proof}

Hence we can say that collection  of $ s\lambda $-open sets in $ (X,\mu) $ forms a genealized topology.

\begin{definition} (c.f.\cite{WD})\label{22}.  A GTspace
$ (X,\mu) $ is said to be $ s\lambda T_\frac{1}{2} $  if  every $ sg_\lambda $-closed set is $ s\lambda $-closed.
\end{definition}

\begin{theorem}\label{23}(c.f.\cite{MS}).
In a GTspace $ (X,\mu) $, folloing are equivalent:

(1) $ (X,\mu) $ is $ s\lambda T_\frac{1}{2} $ GTspace.

(2) Every singleton of $ X $ is either $ s\lambda $-open or, $ s\lambda $-closed.

(3) Every $ sg\wedge_\lambda $-set is $ s\wedge_\lambda $-set.
\end{theorem}
\begin{proof}
(1) $\Rightarrow$ (2): Let $ (X,\mu) $ be a $ s\lambda T_\frac{1}{2} $ GTspace and $ \{x\}\in X $. By theorem \ref{17} (2), $\{x\}$ is either $ s\lambda $-closed or $ X-\{x\} $ is $ sg_\lambda $-closed. If $ X-\{x\} $ is $ sg_\lambda $-closed then $ X-\{x\} $ is $ s\lambda $-closed since $ (X,\mu) $ is $ s\lambda T_\frac{1}{2} $ and so $ \{x\} $ is $ s\lambda $-open.

(2) $\Rightarrow$ (3): Suppose every singleton of $ X $ is either $ s\lambda $-open or, $ s\lambda $-closed. Let $ A \subset X $ be a $ sg\wedge_\lambda $-set which is not $ s\wedge_\lambda $-set. Then $ sA_\lambda^\wedge\not\subset A$. So there exists $ x\in sA_\lambda^\wedge, x\not\in A $. Then by supposition, $\{x\}$ is either $ s\lambda $-open or $ s\lambda $-closed.

Case (i): If $\{x\} $ is $ s\lambda $-open, then $ X-\{x\} $ is $ s\lambda $-closed containing $ A $. But $ A $ is a $ sg\wedge_\lambda $-set. So $ sA_\lambda^\wedge \subset X-\{x\}$ implies $ x\not\in  sA_\lambda^\wedge $, a contradiction.

Case (ii): If $\{x\} $ is $ s\lambda $-closed, then $ X-\{x\} $ is $ s\lambda $-open containing $ A $. But $ x\in   sA_\lambda^\wedge$. So $ x\in X-\{x\} $, a contradiction. Hence the result.

(3) $\Rightarrow$ (1) : Suppose every $ sg\wedge_\lambda $-set is $ s\wedge_\lambda $-set and $ (X,\mu) $ is not $ s\lambda T_\frac{1}{2} $ GTspace. Then there exists a $ sg_\lambda $-closed set $ A $ which is not $ s\lambda $-closed. Since $ A $ is not $ s\lambda $-closed, there is a point $ x\in \overline{sA_\lambda} $ such that $ x\not\in A $. By Theorem \ref{17} (1), $ \{x\} $ is either $ s\lambda $-open or, $ X-\{x\} $ is a $ sg\wedge_\lambda $-set. Suppose $\{x\} $ is  $ s\lambda $-open. Since $ x\in \overline{sA_\lambda}, \{x\}\cap A\not=\emptyset$ which implies $ x\in A $. If $ X-\{x\} $ is a $ sg\wedge_\lambda $ set, then it is $ s\wedge_\lambda $-set, so $ X-\{x\} $ is a $ s\lambda $-open set containing $ A $. Since $ A $ is $ sg_\lambda $-closed,  $\overline{sA_\lambda}\subset X-\{x\} $ which contradicts $ x\in \overline{sA_\lambda} $. So $ x\in A $, hence $ A $ is $ s\lambda $-closed and  $ (X,\mu) $ is $ s\lambda T_\frac{1}{2} $ GTspace.
\end{proof}

\begin{theorem}\label{24}  Suppose $ A \subset (X,\mu) $, then we have the following theorems for $s \lambda $-open sets the proof of which can be deduced easily from the definition.

(1)  $ A $    is  $s \lambda $-open if and only if  $ A = N\cup H $,  $ N $ is a $s \vee_\mu $-set  and $ H $ is a $ s\mu $-open set.

(2)  $ A $   is  $s \lambda $-open if and only if  $ A = sA_\mu^\vee \cup sInt_\mu(A) $.
\end{theorem}

Clearly in $ (X,\mu), s \vee_\mu $-sets are $s \lambda $-open sets and $ s\mu $-open sets are $s \lambda $-open sets.

\section{\bf   $s\lambda T_\frac{1}{4}$ GTspace,  $s\lambda T_\frac{3}{8}$ GTspace and  $s\lambda$-homeomorphism}

In this section we introduce $s\lambda T_\frac{1}{4}$ GTspace,  $s\lambda T_\frac{3}{8}$ GTspace and investigate some of their properties including relation with $s\lambda T_0, s\lambda T_1 $ and $ s\lambda T_\frac{1}{2} $ and discuss $s\lambda$-homeomorphism.

\begin{definition}\label{25} Suppose $(X, \mu)$ is a GTspace, then 

(1)  it  is called  $ s\lambda T_0 $ if for any two distinct points $ x,y $ of $ X $, there exists a $ s\lambda $-open set $ U $ which contains only one of the points.

(2) it is called  $ s\lambda T_1$ if for any two distinct points $x, y \in X$, there are $ s\lambda $-open sets  $ U, V $ such that $x \in U,  y \not \in U,  y\in V , x\not\in V$.
\end{definition}

\begin{theorem}\label{26}
$(X,  \mu)$  is $s\lambda T_0$  if and only if for any pair of distinct points $x, y \in X $,  there is a set $A$ containing only one of the points such that $ A $ is either $ s\lambda $-open or $ s\lambda $-closed.
\end{theorem}
\begin{proof}
Suppose $ x,y \in X, x\not=y$. Let $ A$ be a $ s\lambda $-open set,  $ x\in A, $ and $ y\not\in A $. Then by definition, $(X, \mu)$ is a $ s\lambda T_0 $ GTspace. Now let $ x\in A, y\not\in A $ but $ A $ is $ s\lambda $-closed. Then $ X-A $ is $ s\lambda $-open and $ y\in X-A, x\not\in X-A $. In this case also the GTspace $(X, \mu)$ is a $ s\lambda T_0 $.

Conversely, let the GTspace $(X, \mu)$ be  $ s\lambda T_0 $  and $ x,y \in X, x\not=y $. Then there is a $ s\lambda $-open set $ A $ such that $ x\in A, y\not\in A $ or there is a $ s\lambda $-open set $ B $ such that $ y\in B, x\not\in B $. So $ x\in X-B, y\not\in X-B $ when $ X-B $ is a $ s\lambda $-closed set. Hence the result follows.
\end{proof}

\begin{theorem}\label{27}
A GTspace $ (X,\mu) $ is $ s\lambda T_1$  if and only if every singleton of $ X $ is $ s\lambda $-closed.
\end{theorem}
\begin{proof}
Let the GTspace $ (X,\mu) $ be $ s\lambda T_1$  and $ x\in X $. Take any point $ y\in X$ such that $ x\not=y $. Then there exists $ s\lambda $-open set $ V $ containing $ y $ such that $ x\not\in V $. So $ y $ can not be a $ s\lambda $-adherence point of $ \{x\} $. Therefore $ \overline{s\{x\}_\lambda}=\{x\} $ and hence $ \{x\} $ is $ s\lambda $-closed.

Conversely, let every singleton of $ X $ be $ s\lambda $-closed and let $ x,y\in X, x\not=y $. So $ X-\{x\} $ and $ X-\{y\}$ are $ s\lambda $-open sets such that $ y\in X-\{x\}, x\not\in X-\{x\} $ and $ x\in X-\{y\}, y\not\in X-\{y\} $. Hence the GT space $ (X,\mu) $ is $ s\lambda T_1$.
\end{proof}

\begin{definition}\label{28}
A set $ A $ of the GTspace $ (X,\mu) $  is said to be $ s\beta_\lambda $-closed if $ A=H\cap Q $ where $ H $ is a $ s\wedge_\lambda $-set and $ Q $ is a $ s\lambda $-closed set. $ A $ is $ s\beta_\lambda $-open if $ X-A $ is $ s\beta_\lambda $-closed. 
\end{definition}

\begin{lemma}\label{29}
A set $ A $ of $ (X,\mu) $ is $ s\beta_\lambda $-closed if and only if $ A=sA_\lambda^\wedge \cap \overline{sA_\lambda} $.
\end{lemma}
\begin{proof}
Let $ A $ be a $ s\beta_\lambda $-closed set, so $ A=H\cap Q $ where $ H $ is a $ s\wedge_\lambda $-set and $ Q $ is a $ s\lambda $-closed set. Now $ A\subset H $ implies $ sA_\lambda^\wedge\subset sH_\lambda^\wedge $ and $ A\subset Q $ implies $ \overline{ sA_\lambda}\subset \overline{ sQ_\lambda} $. So $ A\subset sA_\lambda^\wedge\cap\overline{sA_\lambda}\subset sH_\lambda^\wedge\cap\overline{sQ_\lambda}=H\cap Q= A $ which implies that $ A=sA_\lambda^\wedge \cap \overline{sA_\lambda} $.

Converse part is obvious.
\end{proof}

\begin{remark}\label{29A}
Clearly in $ (X,\mu),  s\lambda $-closed set is both $ sg_\lambda $-closed and $ s\beta_\lambda $-closed but either of two may not imply $ s\lambda $-closed. Here is given necessary and sufficient condition for that.
\end{remark}

\begin{theorem}\label{30}
A set $ A \subset X$ is $ s\lambda $-closed if and only if $ A $ is both $ sg_\lambda $-closed and $ s\beta_\lambda $-closed.
\end{theorem}
\begin{proof}
Suppose $A$ is $ sg_\lambda$-closed  and $ s\beta_\lambda$-closed. Since $ A $ is $s g_\lambda $-closed,  then $ \overline{sA_\lambda}\subset  sA_\lambda^\wedge $.   Again since $A$ is $s \beta_\lambda$-closed,  $A=sA_\lambda^\wedge \cap\overline{sA_\lambda}=\overline{sA_\lambda}$. Hence $ A $ is $ s\lambda $-closed. 

Converse part is obvious.
\end{proof}

\begin{theorem}\label{31}
A GTspace $ (X,  \mu) $ is $s\lambda T_\frac{1}{2} $ if and only if  every subset of $(X,\mu)$ is $s\beta_\lambda$-closed.
\end{theorem}

\begin{proof} 
 Suppose every subset of  $(X,\mu)$ is $s\beta_\lambda$-closed and $ x\in X, \{x\}$ is not $s\lambda $-open.  Then $X-\{x\}$ is not $ s\lambda $-closed but $ \overline{s(X-\{x\})_\lambda}=X $. By assumption, $X-\{x\}$ is a $s\beta_\lambda$-closed set, then by lemma \ref{29}, $X-\{x\} = s(X-\{x\})_\lambda^\wedge\cap \overline{s(X-\{x\})_\lambda}$ = $s(X-\{x\})_\lambda^\wedge\cap X = s(X- \{x\})_\lambda^\wedge$. Therefore $X-\{x\}$ is a $s\wedge_\lambda $-set. So $X-\{x\}$ is a $ s\lambda $-open set  which implies that $\{x\}$ is a $s\lambda$-closed set. Then by theorem \ref{23} (2), $(X,\mu)$ is $s\lambda T_\frac{1}{2}$ GTspace.
 
 Conversely, suppose that $ (X, \mu) $ is $s\lambda T_\frac{1}{2}$ GTspace and $ A\subset X $.  Then by theorem \ref{23} (2), every singleton is either $ s\lambda $-open or $ s\lambda $-closed.  Put $ G=\bigcap \{X - \{x\} : x\in X-A, \{x\}$ is $ s\lambda $-closed\} and   $H=\bigcap\{X - \{x\} : x\in X-A, \{x\}$ is $ s\lambda $-open\}. 
 Therefore, $ G $ is a $s \wedge_\lambda $-set and $ H $ is a $ s\lambda $-closed set and $ A=G\cap H $. Hence $ A $ is a $s\beta_\lambda$-closed set.
 \end{proof}

\begin{definition}\label{32}(c.f.\cite{MS}). Suppose $ (X,  \mu) $ is a GTspace then it is called a

(1)   $s\lambda T_\frac{1}{4}$ GTspace if for every finite subset $E$ of $X$ and for every $ y\in X - E $,  there exists a set $G_y$ containing $E$ and $G_y\cap \{y\}=\emptyset$  such that $G_y$ is either $ s\lambda $-open or $ s\lambda $-closed.
 
(2)  $s\lambda T_\frac{3}{8}$ GTspace if for every countable subset $E$ of $X$ and for every $ y\in X - E$,  there exists  a set $ G_y $ containing $E$ and $G_y\cap \{y\}=\emptyset$ such that $ G_y $ is either $ s\lambda $-open or $ s\lambda $-closed.
\end{definition}

Clearly if we take $ E=\{x\} $ in the definition \ref{32} (1), then in view of theorem \ref{26} we see that every $s\lambda T_\frac{1}{4}$ GTspace is a $s\lambda T_0$.

\begin{theorem}\label{33} Let $ (X, \mu) $ be a GTspace, then the following results hold.

(1)  $ (X,\mu) $ is $s\lambda T_0 $ if and only if every singleton of $X$ is $s\beta_ \lambda$-closed.

(2) $ (X,\mu) $ is $ s\lambda  T_\frac{1}{4}$ if and only if every finite subset of $X$ is $s \beta_\lambda$-closed.

(3) $ (X,\mu) $ is $s\lambda T_\frac{3}{8}$ if and only if every countable subset of $X$ is $ s\beta_\lambda$-closed.
\end{theorem} 

\begin{proof} 
We prove the result (2) only; proofs of other are are similar to that of (2) and so is omitted. It can be said from theorems \ref{31} and \ref{33} (1) that $ s\lambda T_\frac{1}{2}$ GTspace implies $ s\lambda T_0$.

Suppose $ (X,  \mu) $ is $s\lambda T_\frac{1}{4}$ GTspace  and $E$ is a finite subset of $X$.  So for every $ y\in X - E $ there is a set $ G_y $ containing $E$ and disjoint from $ \{y\} $ such that $ G_y $ is either $ s\lambda $-open or $ s\lambda $-closed.  Let $K$ be the intersection of all such $ s\lambda $-open sets $ G_y $ and $P$ be the intersection of all such  $ s\lambda $-closed sets  $G_y$ as $ y $ runs over $ X -E $.  Then  $K$ is a $ s\wedge_\lambda $-set  and $ P $ is a $ s\lambda$-closed set and $ E=K\cap P $. Thus $ E $ is $s \beta_\lambda$-closed.

Conversely,  let $E$ be a finite set of $ X $ such that  $ E $ is $s \beta_\lambda$-closed.  Let  $ y\in X - E $. Then $ E=K\cap P $ where $K$ is a $s \wedge_\lambda $-set and $ P $ ia a $ s\lambda $-closed set. If $ y\not\in P $ then the case is obvious since $ P=\overline{sP_\lambda} $. If $ y\in P $, then $ y\not\in K $. Then there exists some $ s\lambda $-open set $ U $ containing $E$ such that $ y\not\in U $.  Hence $ (X,  \mu) $ is $s\lambda T_\frac{1}{4}$ GTspace.
\end{proof}

\begin{remark}\label{34}
It follows from theorems \ref{31}, \ref{33} (3),   \ref{33} (2) that every $s\lambda T_\frac{1}{2}$  GTspace implies $s\lambda T_\frac{3}{8}$ and every $s\lambda T_\frac{3}{8}$ GTspace implies  $s\lambda  T_\frac{1}{4}$. 
\end{remark}

\begin{theorem}\label{35}
If the GTspace $ (X, \mu )$ is  $s\lambda T_0 $  then for every pair of distinct points $ p,q\in X $,  either $ p\not \in\overline{s\{q\}_\lambda} $  or $ q\not \in \overline
{s\{p\}_\lambda} $.
\end{theorem}
\begin{proof}
Let the GTspace $ (X, \mu) $ be  $s\lambda T_0 $  and $ p, q\in X, p \not= q $. Then there exists a $ s\lambda $-open set $ U $ which contains only one of $ p, q $. Suppose that $ p\in U $ and $ q \not \in U $. Then the $ s\lambda $-open set $ U $ has an empty intersection with $ \{q\} $. Hence $  p\not\in\overline{s\{q\}_\lambda}$. Similarly if $ U $ contains the point $ q $  but not $ p $ then $ q\not\in \overline{s\{p\}_\lambda} $. 
\end{proof}

\begin{definition}\label{36}
A GTspace $ (X, \mu) $ is said to be $ s\lambda $-symmetric if $ x,  y \in X,  x \in \overline{s\{y\}_\lambda} \Rightarrow y\in \overline{s\{x\}_\lambda} $. 
\end{definition}

\begin{theorem}\label{37}
A GTspace $ (X,\mu) $ is $ s\lambda $-symmetric iff $ \{x\} $ is $ sg_\lambda $-closed for each $ x\in X $.
\end{theorem}
\begin{proof}
Suppose each singleton of $ X $ is $ sg_\lambda $-closed  and  for $ x,  y \in X $,  let  $x \in \overline{s\{y\}_\lambda} $ but $ y\not\in \overline{s\{x\}_\lambda} $. Then  $\{y\}\subset  X-\overline{s\{x\}_\lambda} $, a $ s\lambda $-open set. Since $ \{y\} $ is  $ sg_\lambda $-closed we get $ \overline{s\{y\}_\lambda} \subset X-\overline{s\{x\}_\lambda} $ and so $ x\in X-\overline{s\{x\}_\lambda} $, a contradiction. Hence $ y\in  \overline{s\{x\}_\lambda} $.

Conversely, suppose $ x\in X,  \{x\}\subset U $, a $ s\lambda $-open set, but $ \overline{s\{x\}_\lambda}\not \subset U $. This implies that $ \overline{s\{x\}_\lambda}\cap (X-U)\not = \emptyset $. Take $ y\in  \overline{s\{x\}_\lambda}\cap (X-U) $. For  $ s\lambda $-symmetryness, $ x\in \overline{s\{y\}_\lambda}$. As $ y\in X-U $ so  $ \overline{s\{y\}_\lambda}\subset X-U$ and $ x\not\in U $. This is a contradiction. Hence the result.  
\end{proof}
 
\begin{theorem}\label{38}
A GTspace $ (X,\mu) $ is $ s\lambda T_1 $ if and only if it is $ s\lambda $-symmetric  and $ s\lambda T_0 $.
\end{theorem} 
\begin{proof}
Let $ (X,\mu) $ be $ s\lambda T_1 $ GTspace. Then obviously it is $ s\lambda T_0 $. Since $ (X,\mu) $ is $ s\lambda T_1 $, by theorem \ref{27}, we get every singleton is $ s\lambda $-closed and so $ sg_\lambda $-closed which by theorem \ref{37}, the GTspace$ (X,\mu) $  is $ s\lambda $-symmetric.

Conversely, Let $ (X, \mu)$ be $ s\lambda $-symmetric and $s\lambda T_0 $ GTspace and let $ x, y \in X,  x\not= y$. Since  $ (X,  \mu)$ is $s\lambda T_0 $, by theorem \ref{35} either $ x\not\in\overline{s\{y\}_\lambda} $ or $ y\not\in\overline{s\{x\}_\lambda} $. Let $ x\not\in\overline{s\{y\}_\lambda} $. Then  $ y\not\in\overline{s\{x\}_\lambda} $.  For if $ y\in \overline{s\{x\}_\lambda} $  then it would imply $ x\in\overline{s\{y\}_\lambda} $, since $ (X, \mu) $ is $ s\lambda $-symmetric. Again since $ x\not\in \overline{s\{y\}_\lambda}$, there is a $ s\lambda $-closed set $ F $ such that $ y\in F $ and $ x\not\in F $. So $ x\in X-F $, a $ s\lambda $-open set and $ y\not\in X-F $. Also since $ y\not\in\overline{s\{x\}_\lambda} $, there is a $ s\lambda $-closed set $ P $ such that $ x\in P $ and $ y\not\in P $.  So $ y\in X-P $, a $ s\lambda $-open set and $ x\not\in X-P $.  Hence $( X,  \mu) $ is $s\lambda T_1 $.  
\end{proof}

\begin{theorem}\label{39}
If the GTspace $ (X, \mu) $ is $ s\lambda $-symmetric   then the GTspaces 
$s\lambda T_0 $,    
$s\lambda T_1 $,   
$s\lambda T_\frac{1}{2} $, 
$ s\lambda T_\frac{3}{8} $,   
$s\lambda T_\frac{1}{4} $ are all equivalent. 
\end{theorem}
\begin{proof}
(1): $ s\lambda T_0 $ GTspace with $ s\lambda $-symmetryness is $s\lambda T_1 $ by theorem \ref{38},

(2):  $ s\lambda T_1 $ GTspace is $s\lambda T_\frac{1}{2} $ by theorem \ref{27}, \ref{23} (2), 

(3):  $s\lambda T_\frac{1}{2} $ GTspace is $ s\lambda T_\frac{3}{8} $ by theorem \ref{31}, \ref{33} (3),

(4):  $s\lambda T_\frac{3}{8} $ GTspace is $ s\lambda T_\frac{1}{4} $ by theorem  \ref{33} (3), \ref{33} (2),

(5):  $ s\lambda T_\frac{1}{4} $ GTspace is $ s\lambda T_0 $ by theorem \ref{33} (2), \ref{33} (1).
\end{proof}

\begin{definition}\label{40} (c.f.\cite{AD}). Let $ (X,\mu_1) $ and $ (Y,\mu_2) $ be two GTspaces. Then a function $ f: (X,\mu_1)  \longrightarrow (Y,\mu_2) $ is said to be $ s\lambda $-continuous (respectively $ s\beta_{\lambda} $-continuous and $ sg_{\lambda} $-continuous) if the inverse image of each $ \mu_2 $-open set $ V $ is $ s\lambda $-open (respectively $ s\beta_{\lambda}$-open and $ sg_{\lambda} $-open) in $(X,\mu_1)  $.

Clearly, a function $ f: (X,\mu_1)  \longrightarrow (Y,\mu_2) $ is $ s\lambda $-continuous (respectively $ s\beta_{\lambda} $-continuous and $ sg_{\lambda} $-continuous) if the inverse image of each $ \mu_2 $-closed set $ P $ is $ s\lambda $-closed (respectively $ s\beta_{\lambda}$-closed and $ sg_{\lambda} $-closed) in $(X,\mu_1)  $.
\end{definition}

\begin{theorem}\label{41} Suppose $ (X,\mu_1) $ and $ (Y,\mu_2) $ are two GTspaces.
A function $ f: (X,\mu_1)  \longrightarrow (Y,\mu_2) $ is $ s\lambda $-continuous if and only if $ f $ is both $ s\beta_{\lambda} $-continuous and $ sg_{\lambda} $-continuous.
\end{theorem}
\begin{proof}
Suppose a function $ f: (X,\mu_1)  \longrightarrow (Y,\mu_2) $ is $ s\lambda $-continuous and $ V $ is a $ \mu_2 $-closed set. Then $ f^{-1}(V) $ is $ s\lambda $-closed in $(X,\mu_1)  $. Hence by remark \ref{29A}, $ f^{-1}(V) $ is both $ s\beta_{\lambda} $-closed and $ sg_{\lambda} $-closed in $(X,\mu_1)  $. Hence $ f $ is both $ s\beta_{\lambda} $-continuous and $ sg_{\lambda} $-continuous.

Conversely, suppose $ f $ is both $ s\beta_{\lambda} $-continuous and $ sg_{\lambda} $-continuous and $ V $ is a $ \mu_2 $-closed set. Then $ f^{-1}(V) $ is both $ s\beta_{\lambda} $-closed and $ sg_{\lambda} $-closed in $(X,\mu_1)  $. So by theorem \ref{30}$,  f^{-1}(V) $ is $ s\lambda $-closed in $(X,\mu_1)  $. Hence the result follows.
\end{proof}

\begin{definition}\label{42} Let $ (X,\mu_1) $ and $ (Y,\mu_2) $ be two GTspaces.
A bijective mapping  $ f: (X,\mu_1)  \longrightarrow (Y,\mu_2) $ is called $ s\lambda $-homeomorphism if the following conditions hold.

(i) Inverse image of every $ s\lambda $-open set in $(X,\mu_2) $ under $ f $ is  $ s\lambda $-open in $(X,\mu_1)  $ and 

(ii) Inverse image of every $ s\lambda $-open set in $(X,\mu_1) $ under $ f^{-1} $ is $ s\lambda $-open in $(X,\mu_2) $.

Note that if $ \mu_1=\mu_2=\mu $ then the conditions (i) and (ii) may be expressed as a single condition (B):  ``Inverse image of every $ s\lambda $-open set in $ (X,\mu) $ is $ s\lambda $-open in $ (X,\mu) $ under $ f $ and $ f^{-1} $."

For a set $ E $ of a GTspace $ (X,\mu) $ we denote the set of all maps $ f: (X,\mu) \longrightarrow (X,\mu) $ such that $ f $ is a $ s\lambda $-homeomorphism and $ f(E)=E $ by  $ s\lambda h(X,E, \mu) $. When $ E=\emptyset $ we denote $ s\lambda h(X,\emptyset, \mu) $ simply by $ s\lambda h(X, \mu) $. 

The composition $ g\circ f $ for every $ f, g \in s\lambda h(X,\mu) $ is clearly a binary operation on $ s\lambda h(X,\mu) $.
\end{definition}

\begin{theorem}\label{44} (c.f.\cite{AD}).
Suppose $(X,\mu) $ is a GTspace. Then   $  s\lambda h(X,\mu) $ is a group and $ s\lambda h(X,E,\mu) $ is a subgroup of $ s\lambda h(X,\mu) $ for each $ E\subset X $ with respect to the composition of maps of $ s\lambda h(X,\mu) $ defined above.
\end{theorem}
\begin{proof}
Clearly the operation is associative and closed. 
Since the identity map satisfies the condition(B), it belongs to  $s\lambda h(X,\mu) $ and it is the identity element of the set. Again if any $ f\in s\lambda h(X,\mu) $, then $ f^{-1} $ satisfies the condition (B) and hence $ f^{-1}\in s\lambda h(X,\mu) $. So $ s\lambda h(X,\mu) $ is a group with this binary operation. 

Next, since $ s\lambda h(X,E,\mu)  \subset s\lambda h(X,\mu)$ then we have to prove only that for any $ f,g \in s\lambda h(X,E,\mu), g\circ f^{-1}\in s\lambda h(X,E,\mu) $. Since  $ f,g \in s\lambda h(X,E,\mu),$ then $ f,f^{-1}, g, g^{-1}  $ satisfy the condition (B) which map $ E $ into $ E $. Hence $ g\circ f^{-1} $ and $ f\circ g^{-1} $ satisfy the condition (B). Again, $ g\circ f^{-1}(E)=g(E)=E $. Therefore $ s\lambda h(X,E,\mu) $ is a subgroup of $ s\lambda h(X,\mu) $ for each $ E\subset X $.
\end{proof}

\end{document}